\documentclass[11pt, twoside]{article}
\usepackage{amsmath,amsthm,amssymb}
\usepackage{times}
\usepackage{enumerate} 

\usepackage{color}
\usepackage[colorlinks]{hyperref} 

\makeatletter
\def\author#1{\gdef\autrun{\def\and{\unskip, }#1}\gdef\@author{#1}}

\makeatother

\newtheorem{thm}{Theorem}[section]

\newtheorem{lemma}[thm]{Lemma}
\newtheorem{prop}[thm]{Proposition}

\theoremstyle{definition}

\newtheorem{rem}[thm]{Remark}

\numberwithin{equation}{section}

\def\N{\mathbb {N}}

\def\Z{\mathbb {Z}}

\def\C{\mathbb {C}}
\def\P{\mathbb {P}}

\def\B{\mathcal {B}}

\def\D{\mathbb {D}}

\def\XX{\mathring X}
\def\YY{\mathring Y}
\def\ff{\mathring f}
\def\uu{\mathring u}

\def\TT{\mathring T}
\def\QQQ{\mathring Q}
\def\ZZ{\mathring Z}
\def\SS{\mathring S}

\def\RR{\mathring R}

\def\pp{\mathring p}
\def\tt{\mathring \tau}

\begin{document}


\baselineskip=17pt

\title{Non trivial limit distributions for transient renewal chains}
\author{Dalia Terhesiu\footnote{
University of Exeter,
North Park Road, Exeter,
UK, EX4 4QF.
Email: daliaterhesiu@gmail.com.}}

\date{December, 2016}

\maketitle

\begin{abstract}
In this work we study the asymptotic of renewal sequences associated with certain transient renewal Markov chains
and enquire about the existence of limit laws in this set up.
\end{abstract}

\section{Introduction}

In the first part of this work we are interested in the asymptotic behaviour of renewal sequences
associated with transient Markov  renewal chains with regularly varying tails of the return time to the state $[0]$.
The precise meaning of the transient renewal chains considered here is given in  Subsection~\ref{subsec-trans} (in particular, see equation~\eqref{eq-tail}).
In this set up, we show that up to a constant, independent of the index of regular variation, the renewal sequences are asymptotically equivalent
to the tails of the return to state $[0]$: see Proposition~\ref{prop-main2} (and its weaker version Proposition~\ref{prop-main}) in Section~\ref{sec-limrenseq}.
The result in Proposition~\ref{prop-main2} is implicit in the work~\cite{DoneyKor}, which focuses on transient random walks on $\Z^d$, $d\ge 1$. In short, Proposition~\ref{prop-main2} is a result of similar flavour
to that in~\cite[Theorem 4]{DoneyKor}.
The analytical proof of~\cite[Theorem 4]{DoneyKor}(in its full generality) in~\cite[Section 2]{DoneyKor}
relies on~\cite[Theorem 1]{Choverall}, of which proof is based on deep Banach algebra techniques. The proof of the present  Proposition~\ref{prop-main2} is entirely elementary.

In the second part, restricting to indices of regular variation that, provided that the renewal chain is recurrent, would imply it is null recurrent, we enquire about the existence of limit laws.
The main result of this paper,  Theorem~\ref{thm-arc} in Section~\ref{sec-arcsine}, shows the existence of an arcsine law for the transient chain; the proof of this result exploits the asymptotic
behaviour of the renewal sequence obtained in Proposition~\ref{prop-main2}.
In Section~\ref{sec-ratio}, we provide an asymptotic characterization of the random variable $S_n$ describing the number of visits to the state $[0]$
in the interval $[0,n]$ when appropriately scaled: see Proposition~\ref{prop-surv}.

We believe that the techniques in this work can be extended to dynamical systems, in which any form of independence fails. 
Typical systems that, apart from independence,
resemble a renewal chain are the so called interval maps with indifferent fixed points such as the one studied in~\cite{LiveraniSaussolVaienti99}. The task of extending the present
results to such systems is beyond the scope of this work, but once accomplished it could offer an alternative to the results in~\cite{DemersFernandez}.

\section{Set up. Notation}

\subsection{Renewal chain, induced renewal chain}
\label{subsec-renrec}

Let  $(X_n)_{n \geq 0}$, $X_n \in \N_0=\N\cup \{0\}$
be a Markov renewal chain
with transition probabilities
$$
p_{\ell, k} := \P(X_{n+1} = k | X_n = \ell) = 
\begin{cases}
 f_k & \ell = 0,\\
 1 & k = \ell-1,\\
 0 & \text{otherwise.}
\end{cases}
$$
We assume $\sum_k f_k=1$ and recall that depending on the asymptotics of $\sum_{k>n}f_k$,  $(X_n)_{n \geq 0}$ is a positive recurrent  or a null-recurrent renewal chain
(see, for instance,~\cite{Feller66}).

Let $X = \N_0^{\N_0}$ and let  $T:X \to X$ be the shift map.
Then any cylinder $[e_0e_1...e_{k-1}]$ has measure
$\mu([e_0e_1...e_{k-1}]) = \mu([e_0]) p_{e_0e_1} \cdots p_{e_{k-2} e_{k-1}}$.
This can be computed if the initial distribution $\mu([j]), j \in \N_0$,
is given. The Markov measure $\mu$ is $T$ invariant.

Let $Y = [0] = \{ x \in X : x_0 = 0\}$, and decompose
$$
Y = \cup_{k \geq 0} C_k, \quad \text{ where } C_k = [0,k,k-1,k-2,\dots, 0].
$$
The cylinders $C_k$ are pairwise disjoint, and their
measures are given by
\[
\mu(C_{k}) = \mu(Y) p_{0,k} p_{k,k-1} \cdots p_{1,0} = \mu(Y) f_k.
\] 

We recall the definition of the \emph{induced shift on $Y$} and associated 'induced renewal chain'.
For $y\in Y$, let $\tau(y) = \min\{ n \geq 1 : T^n(y) \in Y\}$ and $T_Y = T^\tau$. The probability measure
$\nu=\mu(Y)^{-1}\mu|_Y$ is  $T_Y$ invariant. We note that 
$C_k = \{ y \in Y : \tau(y) = k+1\}$ can be regarded as the shift on the space 
$(\{ C_k \}_{k\ge 0})^{\N_0}$.

Define the induced Markov chain
$(U_n)_{n \geq 0}$, $U_n \in \{ C_k \}_{k \geq 0}$,
with transition probabilities:
\begin{eqnarray}
\label{eq-matrix-indh}
\nonumber \hat p_{\ell, k} &=& \P(U_{n+1} = C_k | U_n = C_\ell)
= \frac{\P(U_{n+1} = C_k \wedge U_n = C_\ell)}{\P(U_{n+1} = C_k)} =  \\
&=& \frac{ \P(C_\ell \wedge T_Y^{-1} (C_k)) }{ \P(U_{n+1} = C_k) } = 
\frac{ p_{0,\ell} p_{\ell,\ell-1} \cdots p_{1,0} p_{0,k}}{p_{0,\ell}}
= f_k.
\end{eqnarray}
Note that $\hat p_{\ell, k}$ is independent of $\ell$.

The induced renewal chain $(U_n)_{n \geq 0}$ with above
transition probabilities $\hat p_{\ell, k}$ is positive recurrent.
To see this, fix $k\ge 1$
and let $\nu^* = \frac{1}{\mu(C_k)} \mu|C_k$ and $\varphi:C_k\to \N$
with $\varphi(y) := \min\{ n \geq 1 : T_Y^n(y) \in C_k\}$ 
be the first return time of $T_Y$ to $C_k$.
Since $\P(U_n = k) = \sum_{\ell} \P(U_n = k | U_{n-1} = \ell) \P(U_{n-1} = \ell)
= \sum_{\ell} \hat p_{\ell,k} \P(U_{n-1}=\ell) =  \sum_{\ell} f_k \P(U_{n-1}=\ell) = f_k$, we have
\begin{eqnarray*}
\nu^*(\varphi \geq n) &=&
\frac{1}{\mu(C_k)} \sum_{m \geq n} \mu(y \in C_k : \varphi(y) = m) \\
&=&
\frac{1}{\mu(C_k)} \sum_{m \geq n} \P(U_0 = C_k \wedge
U_j \neq C_k, 0 < j < m \wedge U_m = C_k) \\
&=&
\frac{1}{\mu(C_k)} \sum_{m \geq n} \mu(C_k) (1-f_k)^{m-1} f_k = (1-f_k)^n.
\end{eqnarray*}
Hence $\varphi$ has an exponential distribution, which shows that $(U_n)_{n \geq 0}$ is positive recurrent (since  $\sum_n\nu^*(\varphi \geq n)<\infty$).

\subsection{Introducing transience, 'holes' in the the original chain $(X_n)_{n \geq 0}$}
\label{subsec-trans}

Recall that $T:X\to X$ is the original shift
and $T_Y=T^\tau:Y\to Y$ is the induced shift with $Y = [0] = \{ x \in X : x_0 = 0\}$. Throughout we assume that
\begin{align}
\label{eq-ap}
g.c.d. \{\tau|_{C_k}, k\geq 0\}=1,
\end{align}
which ensures that $(X_n)_{n \geq 0}$, $X_n \in \N_0$ is aperiodic.

We introduce a hole $H$ in $X$ with $H\subset Y$ and thus transience\footnote{This type of rule for introducing transience/holes in Markov chains was suggested to me by Roland Zweim{\"u}ller. In particular, the results in
Section~\ref{sec-limrenseq} answer his questions. I wish to thank him for useful discussions on this topic.}, as follows. Let $\mathring{X}=X\setminus H$ and $\mathring{Y}=Y\setminus H$. Set ${\XX}^n=\cap_{i=0}^n T^{-i}\XX$ and define  $\TT=T|_{{\XX}}$
s.t. the first return time $\tt$ of $\TT$ to $\YY$ satisfies
\begin{equation}\label{eq-tail}
\nu(\tt=n)=p\nu(\tau=n)=p f_{n-1}:=\ff_{n-1}, n\ge 1.
\end{equation}
Here we recall that $\nu=\mu(Y)^{-1}\mu|_Y$ is the $T_Y$ invariant probability measure. In fact, due to the rule above (of introducing a hole in $X$), $\nu$ is also $\TT_{\YY}=\TT^{\tt}$ invariant. To see this, let $Q_Y$ be the transition matrix
for induced renewal chain $(U_n)_{n \geq 0}$ and note that this is an infinite matrix with $(f_0,f_1,f_2,\dots)$ in every row. Given the set up of the previous subsection, $\nu$ is the left eigenvector
of $Q_Y$ (with eigenvalue $1$). But, the transition matrix for the modified chain (after introducing a hole)  is simply $\QQQ_Y=p Q_Y$. While the eigenvalue changes from $1$ to $p$,
the left eigenvector $\nu$ remains same.

In what follows we are interested in the asymptotics of the renewal sequence associated with the transient renewal chain $(\XX_n)_{n \geq 0}$, $\XX_n \in \N_0$
with transition probabilities
\[
\pp_{\ell, k} := \P(\XX_{n+1} = k | \XX_n = \ell) = 
\begin{cases}
 pf_k & \ell = 0,\\
 1 & k = \ell-1,\\
 0 & \text{otherwise.}
\end{cases}
\]

We start by recalling the renewal equation, which can be obtained word by word as in the recurrent case (see, for instance,~\cite{Feller49}).
For $n\in\N$, let $\tt_n=\sum_{j=0}^{n-1}\tt\circ \TT_{\YY}^j$. Recall that the sequence $(\ff_k)_{k\geq 1}$ is defined in~\eqref{eq-tail} and define the renewal sequence
\begin{equation}\label{eq-reniid}
\uu_0=1,\quad \uu_n=\P\Big(\exists\,  k\le n\mbox{ such that } \sum_{j=0}^k\tt_j=n\Big)=\sum_{j=1}^n \ff_j \uu_{n-j}.
\end{equation}

For $z\in\bar\D$, set $\ff(z)=\sum_1^\infty \ff_n z^n$ and $\uu(z)=\sum_0^\infty \uu_n z^n$. Since, by assumption $\sum_{k\ge 1} \ff_k=p<1$ and~\eqref{eq-ap} holds, we have
that
\begin{equation}\label{eq-reniid2}
\uu(z)=(1-\ff(z))^{-1}
\end{equation}
is well defined on the whole of $\bar\D$.

\section{Non trivial limits for the renewal sequence $\uu_n$}
\label{sec-limrenseq}

The first result below gives the asymptotics of the tail renewal sequence, that is $\sum_{j>n}\uu_j$, where $(\uu_j)_{j\ge 1}$ is the renewal sequence associated
 with the chain $(\XX_n)_{n \geq 0}$, $\XX_n \in \N_0$ introduced in Subsection~\ref{subsec-trans}.
 Throughout this section, we assume the set up of Subsection~\ref{subsec-trans}, in particular~\eqref{eq-tail}  and suppose that ~\eqref{eq-ap} holds.

\begin{prop}\label{prop-main} 
Suppose that $f_n=O(n^{-(\beta+1)})$, for some $\beta>0$.
Then
\[
\sum_{j>n}\uu_j = (1-p)^{-2}\sum_{j>n}\ff_j(1+o(1))=  p  (1-p)^{-2}\nu(\tau>n)(1+o(1)).
\]
\end{prop}

\begin{proof} Compute that
\begin{align*}
\uu(z)-\uu(1)=\sum_{n=0}^\infty\uu_n (z^n-1)=(z-1)\sum_{n=0}^\infty(\sum_{j>n}\uu_j)z^n.
\end{align*}
Together with~\eqref{eq-reniid2}, the above equation gives
\begin{align*}
\sum_{n=0}^\infty(\sum_{j>n}\uu_j)z^n& =(z-1)^{-1}(1-\ff(z))^{-2}(\ff(z)-\ff(1))\\
&=(z-1)^{-1}(1-\ff(1))^{-2}(\ff(z)-\ff(1))\\
&+(1-\ff(1))^{-3}A(z)\Big(1-(1-\ff(1))^{-1}(\ff(z)-\ff(1))\Big)^{-2}\\
&=(1-p)^{-2}\sum_1^\infty (\sum_{j>n}\ff_j) z^n+(1-p)^{-3}A(z)\Big(1-(1-p)^{-1}(\ff(z)-\ff(1))\Big)^{-2},
\end{align*}
where $A(z)=C(z-1)^{-1}(\ff(z)-\ff(1))^2$, for $C>0$.
By Lemma~\ref{lemma-abstr}, the coefficients of $A(z)$ are $o(\sum_{j>n}\ff_j)$.
By Wiener's lemma, the coefficients of $(1-\ff(z))^{-1}$, and thus of
$\Big(1-(1-p)^{-1}(\ff(z)-\ff(1))\Big)^{-2}$, are $O(n^{-(\beta+1)})$.
Convolving, we obtain that the coefficients of  $A(z)(1-(1-\ff(1))^{-1}(\ff(z)-\ff(1))^{-2}$
are $o(\sum_{j>n}\ff_j)$. The conclusion follows.
\end{proof}

The next result gives the asymptotics of $\uu_n$ under a stronger assumption on the asymptotic behaviour of $f_n$.

\begin{prop}\label{prop-main2}
Suppose that $nf_n=C\sum_{j>n}f_j(1+o(1))$ and that $f_n=O(n^{-(\beta+1)})$,  for some $C>0$ and $\beta>0$.
Then
\[
\uu_n = (1-p)^{-2}\ff_n(1+o(1))= p  (1-p)^{-2}\nu(\tau=n)(1+o(1)).
\]
\end{prop}

\begin{rem}
 The above assumption holds under the assumption of regular variation for the sequence $f_n=\nu(\tau=n)$,
that is if $f_n=\ell(n)n^{-(\beta+1)}$ for $\ell$ a slowly varying function.
\end{rem}

\begin{proof} By definition $\uu_n$ is the coefficient of $(1-\ff(z))^{-1}$, so it is $n^{-1}C_n$,
where $C_n$ is the coefficient of $C(z)=\frac{d}{dz}((1-\ff(z))^{-1})$. Compute that
\begin{align*}
 C(z)=(1-\ff(z))^{-2}\frac{d}{dz}(\ff(z))&=(1-\ff(1))^{-2}\Big(\sum_1^\infty n\ff_{n+1} z^n+(\sum_0^\infty \ff_{n+1} z^n\Big)\\
 &+
\tilde C\Big (1-\ff(1))^{-3}B(z)\big(1-(1-\ff(1))^{-1}(\ff(z)-\ff(1))\Big)^{-2},
\end{align*}
where $\tilde C>0$ (independent of $p$) and $B(z)=(\ff(1)-\ff(z))\sum_0^\infty n\ff_n z^n$ . 

Put $D(z)=\tilde C\Big (1-\ff(1))^{-3}B(z)\big(1-(1-\ff(1))^{-1}(\ff(z)-\ff(1))\Big)^{-2}=\sum_0^\infty D_nz^n$ and note that
\[
C_n=n\ff_{n+1}+\ff_{n+1}+O(D_n).
\]
By assumption, $nf_n=C\sum_{j>n}f_j(1+o(1))$. We claim that $D_n=o(\sum_{j>n}\ff_j)$ and the conclusion follows.

To prove the claim we note that by Wiener's lemma,
the coefficients of $\Big(1-(1-\ff(1))^{-1}(\ff(z)-\ff(1))\Big)^{-2}$ are $O(n^{-(\beta+1)})$.
Hence, it suffices to show that
the coefficients $B_n$ of $B(z)$ are $o(\sum_{j>n}\ff_j)$.

Since $nf_n=C\sum_{j>n}f_j(1+o(1))$, we have $n\ff_n=C\sum_{j>n}\ff_j(1+o(1))$. Thus, using the definition of $B(z)$,
\begin{align*}
B(z)&=\frac{(\ff(1)-\ff(z))^2}{z-1}+(\ff(1)-\ff(z))\sum_1^\infty o\Big(\sum_{j>n}\ff_j\Big)z^n.
\end{align*}

By Lemma~\ref{lemma-abstr}, the coefficients of the first term are $o(\sum_{j>n}\ff_j)$. By assumption the coefficients of
$\ff(1)-\ff(z)$ are $O(n^{-(\beta+1)})$ and thus, the coefficients of the second term are $o(\sum_{j>n}\ff_j)$, as required.
~\end{proof}

\section{An arcsine law for $\beta\in (0,1)$}
\label{sec-arcsine}

Recall  that $(\XX_n)_{n \geq 0}$, $\XX_n \in \N_0$ is the transient renewal chain introduced in Subsection~\ref{subsec-trans}
with associated shift $\TT:\XX\to \XX$.
Proposition~\ref{prop-main2} allows us to obtain the following arcsine law. 
Let
\[
\ZZ_n(x):=\max\{0\leq j\leq n: \TT^{j}(x)\in\YY\},
\]
be the last visit of the orbit of $x$ under the  shift $\TT$ to $\YY$ in the interval $[0,n]$. 
In what follows, $\B(\beta,1-\beta)$ is the standard Beta distribution with parameters $\beta, 1-\beta$. Also,
we let $[\,\,\,]$ denote the integer part.

\begin{thm}
\label{thm-arc}
Assume the setting of Proposition~\ref{prop-main2} with  $\ff_n=Cn^{-(\beta+1)}(1+o(1))$, for some $C>0$. Let
$\beta\in (0,2)$ and set  $q=1/(1+2\beta)$. Then
 \[
  \Big(\frac{\ZZ_{[n^q]}}{n}\Big)^{1/q}\to^{\nu_0}\B(\beta,1-\beta),
 \]
where the convergence
is in measure, for any probability measure absolutely continuous w.r.t. $\nu_0= C^{-2}q p^{-1}(1-p)^{2}\nu$.
\end{thm}

\begin{proof}
Let $\hat Z_n(x):=\max\{0\leq j\leq n: \TT^{[j^q]}(x)\in\YY\}$ and note that
\begin{align}
 \label{eq-ZZ}
\nonumber (\ZZ_{[n^q]})^{1/q}=\max\{[j^{1/q}]:0\leq j\leq [n^{q}]: \TT^{j}(x)\in\YY\}&=\max\{j\in\{0,\ldots,n\} :\TT^{[j^q]}(x)\in\YY\}\\
 &=\hat Z_{n}(x).
 \end{align}

But for any $t>0$,
\[
\nu\Big(\frac{\hat Z_{n}(x)}{n^{1/q}}\le t\Big)=\nu(\hat Z_{n}(x)\leq (nt)^{1/q})=\sum_{0\leq j\leq (nt)^{1/q}}\nu(\TT^{[j^q]}\in\YY\cap\{\tt>n-[j^q]\}).
\]

Due to independence,
\begin{align*}
 \nu(\TT^{[j^q]}\in\YY\cap\{\tt>n-[j^q]\})=\nu(\{\tt>n-[j^q]\})\nu(\TT^{-[j^q]}\YY).
 \end{align*}
It is easy to see from the definition of the renewal sequence in~\eqref{eq-reniid} that $\nu(\TT^{-[j^q]}\YY)=\uu_{[j^q]}$.
 Proposition~\ref{prop-main2} together with  $\ff_n=Cn^{-(\beta+1)}(1+o(1))$ implies that
$\uu_{[n^q]} = pC(1-p)^{-2}[n^q]^{-(\beta+1)}(1+o(1))$. Putting the above together and using that\footnote{Here, we also use the convention that $j^{-\gamma}=0$ for $j=0$ and $\gamma>0$.}
$\nu(\{\tt>n\})=Cpn^{-\beta}(1+o(1))$,

\begin{align}
\label{eq-arc1}
\nonumber &\sum_{0\leq j\leq (nt)^{1/q}}\nu(\TT^{-[j^q]}\in\YY\cap\{\tt>n-[j^q]\}=C^2p(1-p)^{-2}\sum_{0\leq j\leq (nt)^{1/q}}\frac{1}{[j^q]^{\beta+1}}\frac{1}{(n-[j^q])^\beta}\\
&=C^2p^2(1-p)^{-2}\sum_{0\leq j\leq (nt)^{1/q}}\frac{1}{j^{q(\beta+1)}}\frac{1}{(n-[j^q])^\beta}+O\Big(\sum_{0\leq j\leq (nt)^{1/q}}\frac{1}{j^{2q(\beta+1)}}\frac{1}{(n-j^q)^\beta}\Big)
\end{align}
For the first term, as $n\to\infty$,
\begin{align*}
\sum_{0\leq j\leq (nt)^{1/q}}\nu(\TT^{-j^q}\in\YY\cap\{\tt>n-[j^q]\})&\to C^2p(1-p)^{-2}\int_1^{(nt)^{1/q}}\frac{1}{s^{q(\beta+1)}}\frac{1}{(n-s^q)^{\beta}}ds.
\end{align*}
Recall $q=1/(1+2\beta)$. With the substitution $s^q\to nu$
\begin{align*}
q\frac{1}{n^\beta}\int_1^{(nt)^{1/q}}\frac{1}{s^{q(1+\beta)}}\frac{1}{(1-\frac{s^q}{n})^{\beta}}ds &=\frac{n^{1/q}}{n^{\beta} n^{\beta+1}}\int_{1/n}^{t}\frac{u^{1/q-1}}{u^{\beta+1}(1-u)^{\beta}}du\\
&=\int_{1/n}^{t}\frac{1}{u^{1-\beta}}\frac{1}{(1-u)^{\beta}}du=\int_0^{t}\frac{1}{u^{1-\beta}}\frac{1}{(1-u)^{\beta}}du+O(1/n^\beta).
\end{align*}
For the second term in~\eqref{eq-arc1},  a calculation similar to the one above shows that 
\[
\sum_{0\leq j\leq (nt)^{1/q}}\frac{1}{j^{2q(\beta+1)}}\frac{1}{(n-j^q)^\beta}=O(1/n^\beta).
\]
Putting the above together, as $n\to\infty$, 
\begin{align}
\label{eq-integer1}
\nu\Big(\frac{\hat Z_{n}}{n^{1/q}}\le t\Big)\to C^2 pq^{-1} (1-p)^{-2}\int_0^{t}\frac{1}{u^{1-\beta}}\frac{1}{(1-u)^{\beta}}du.
\end{align}
The above displayed equation together with~\eqref{eq-ZZ} ends the proof for the case $\beta\in (0,1)$ of  the claimed convergence w.r.t. the measure $\nu_0= C^{-2}q p^{-1}(1-p)^{2}\nu$.
The  convergence in measure, for any probability $\nu_0$ absolutely continuous w.r.t. $\nu$, follows since the density of $\nu$ is a constant.~\end{proof}

\section{A ratio limit  for $\beta\in (0,1)$}
\label{sec-ratio}

It is known that for null recurrent renewal shifts $T:X\to X$,
 $X = \N_0^{\N_0}$ with induced shifts $T_Y=T^\tau: Y\to Y$, $Y = [0] = \{ x \in X : x_0 = 0\}$
as recalled in Subsection~\ref{subsec-renrec}, a Darling Kac law for
$S_n(1_Y)=\sum_{j=0}^{n-1}1_Y\circ T^j$ holds under regular
variation of the tail $\nu(\tau>n)$(see, for instance,~\cite{Feller66}). More precisely,
simplifying the assumption on the tail,
if $\nu(\tau>n)=Cn^{-\beta}(1+o(1))$ for some $C>0$
and $\beta\in (0,1)$, then as $n\to\infty$, $C^{-1}n^{-\beta}S_n(1_Y)\to \mathcal{M}_{\beta}$,
where $\mathcal{M}_{\beta}$ is a random variable distributed according to the Mittag Leffler
distribution\footnote{We recall that the Laplace transform of this random variable is given
by $E(e^{z\mathcal{M}_\beta})=\sum_{p=0}^\infty \Gamma(1+\beta)^pz^p/\Gamma(1+p\beta)$ 
for all $z\in\C$. }.
One way of seeing this
is to recall that: a) $\P(\tau_m\ge n)=\P(S_n(1_Y)\le m)$,
where $\tau_m=\sum_{j=0}^{m-1}\tau\circ T_Y^j$; b) under the assumption
$\nu(\tau>n)=Cn^{-\beta}(1+o(1))$, we have that as $m\to\infty$,
$m^{-1/\beta}\tau_m\to C_\beta \mathcal{Y}_\beta$,
where $\mathcal{Y}_\beta$ is a random variable in the domain of a stable law of index $\beta$
and $C_\beta$ is a constant that depends only on $C$ and $\beta$ ; c) $\mathcal{M}_{\beta}=_d\mathcal{Y}_\beta^{-\beta}$.
This type of argument for the proof of a Darling Kac law can be found, for instance, in \cite{Bingham1}, which goes back to~\cite{Feller49}.

In the case of the transient shift $\TT$ introduced in Subsection~\ref{subsec-trans}, the duality rule in point b) above does not hold. Instead, in this section we
will exploit Lemma~\ref{lemma-ratio} below and obtain the following, more or less obvious, limit behaviour on the survivor set:
\begin{prop}
\label{prop-surv} Assume the set up of Subsection~\ref{subsec-trans}, in particular~\eqref{eq-tail}.  Assume that ~\eqref{eq-ap} holds.
Suppose that
$\nu(\tau>n)=Cn^{-\beta}(1+o(1))$ with $\beta\in(0,1)$. Let $\SS_n(1_{\YY})=\sum_{j=0}^{n-1}1_{\YY}\circ \TT^j$. Then, for any $t>0$,
\[
1\le \frac{\lim_{n\to\infty}p^{n^{1/\beta}}\nu(n^{-1/\beta}\SS_n\le t\cap\XX^{n})}{\P(\mathcal{Y}_\beta \le t)}\le 1+p.
\]
\end{prop}

\begin{proof}
Write $\tt_m=\sum_{j=0}^{m-1}\tt\circ\TT_{\YY}^j$. For notational convenience, from here on we write $S_n, \SS_n$ instead of $S_n(1_Y), \SS_n(1_{\YY})$.

By Lemma~\ref{lemma-ratio} for with $[n^\beta t]=m$,
for $t>0$,
\begin{align}
\label{eq-ratio}
\frac{\nu(\tt_{[n^\beta t]}\ge n\cap\YY^{[n^\beta t]})}{\nu(\SS_n\le [n^\beta t]\cap\XX^{n})}=\sum_{k=1}^{[n^\beta t]}p^{[n^\beta t]-k}\nu(S_n=k).
\end{align}
Rewriting the RHS  using $\nu(S_n=k)=\nu(S_n\le k)-\nu(S_n\le k-1)$
\begin{align*}
\sum_{k=1}^{[n^\beta t]}p^{[n^\beta t]-k}\nu(S_n=k)&=\sum_{k=1}^{[n^\beta t]}p^{[n^\beta t]-k}\nu(S_n\le k)-p\sum_{k=0}^{[n^\beta t-1]}p^{[n^\beta t]-k}\nu(S_n\le k)\\
&=\nu(S_n\le  [n^\beta t] )+(1-p)\sum_{k=1}^{[n^\beta t-1]}p^{[n^\beta t]-k}\nu(S_n\le k).
\end{align*}
Thus
for $n$ large enough,
\begin{align*}
\nu(S_n\le  [n^\beta t] )\le \sum_{k=1}^{[n^\beta t]}p^{[n^\beta t]-k}\nu(S_n=k)&\le \nu(S_n\le  [n^\beta t]) +(1-p)\nu(S_n\le  [n^\beta t] )\sum_{k=1}^{[n^\beta t-1]}p^{[n^\beta t]-k}\\
&\le \nu(S_n\le  [n^\beta t])(1+p).
\end{align*}
Equivalently,
\begin{align*}
\nu(\tau_{ [n^\beta t]}\ge n)\le \sum_{k=1}^{[n^\beta t]}p^{[n^\beta t]-k}\nu(S_n=k)\le \nu(\tau_{ [n^\beta t]}\ge n)(1+p).
\end{align*}
Note  that since $\nu(\tau>n)=Cn^{-\beta}(1+o(1))$, for any $t>0$, we have  $\nu(\tau_{ [n^\beta t]}\ge[n^\beta t]^{1/\beta} )\to C_\beta \P(\mathcal{Y}_\beta\ge t^{1/\beta})$.
Since $\nu(\tau_{ [n^\beta t]}\ge[n^\beta t]^{1/\beta})-\nu(\tau_{ [n^\beta t]}\ge n t^{1/\beta})=o(1)$,
\[
\nu(\tau_{ [n^\beta t]}\ge n)\to C_\beta \P(\mathcal{Y}_\beta\ge 1).
\]
Putting together the previous displayed equations, there exists a constant $D_\beta$ that depends only on $C_\beta$
and $\P(\mathcal{Y}_\beta\ge 1)$ such that
\begin{equation}
\label{eq-sum}
1\le D_\beta^{-1}\sum_{k=1}^{[n^\beta t]}p^{[n^\beta t]-k}\nu(S_n=k)\le 1+p.
\end{equation}
Finally,  by Lemma~\ref{lemma-taumsurv}, $\nu(\tt_{[n^\beta t]}\ge n|\YY^{[n^\beta t]})\to \P(\mathcal{Y}_\beta \le t)$
and thus,
\[
p^{-n^{1/\beta}}\nu(n^{-1/\beta}\tt_{[n^\beta t]}\ge n\cap\YY^{[n^\beta t]})\to \P(\mathcal{Y}_\beta \le t).
\]
The conclusion follows by the above equation together with~\eqref{eq-sum} and~\eqref{eq-ratio}.~\end{proof}

For $n,m\in\N$, the result below relates $\SS_n$ to $\tt_m=\sum_{j=0}^{m-1}\tt\circ\TT_{\YY}^j$ and it can be regarded
as an analogue of item b) mentioned at the beginning of this section.
\begin{lemma}
\label{lemma-ratio}
 Assume the set up of Subsection~\ref{subsec-trans}, in particular~\eqref{eq-tail}. Then for all $n,m\in\N$,
\[
\nu(\tt_m\ge n\cap\YY^m)=\nu(\SS_n\le m\cap\XX^{n})\sum_{k=1}^{m}p^{m-k}\nu(S_n(1_Y)=k\}).
\]
~\end{lemma}

\begin{proof} 
Using that in the recurrent case  $\P(\tau_m\ge n)=\P(S_n\le m)$ (for any probability measure $\P$ on $Y$), we compute that
\begin{align*}
\P(\tt_m\ge n\cap\YY^m)&=\P(\tau_m\ge n\cap\YY^m)=\P(S_n\le m\cap\YY^m)=\P(S_n\le m\cap\YY^{S_n}\cap\YY^{m-S_n} )\\
&=\P(\SS_n\le m\cap\XX^{n})\P(y\in Y: S_n(y)< m, T_Y^{S_n(y)}(y)\in \YY^{m-S_n(y)})\\
&=\P(\SS_n\le m\cap\XX^{n})\P(y\in Y:  T_Y^{S_n(y)}(y)\in \YY^{m-S_n(y)})\\
&=\P(\SS_n\le m\cap\XX^{n})\sum_{k=1}^{m}\P(y\in Y:  T_Y^k(y)\in\YY^{m-k}\cap\{y\in Y:S_n(y)=k\}).
\end{align*}
Clearly, the events $\{y\in Y:  T_Y^k(y)\in\YY^{m-k}\}$ and $\{y\in Y:S_n(y)=k\}$ are disjoint. Recalling that
$\nu$ is $T_Y$  invariant, $\nu(\{y\in Y:  T_Y^k(y)\in\YY^{m-k}\})=\nu(\YY^{m-k})=p^{m-k}$.
Thus,
\[
\sum_{k=1}^{m}\P(y\in Y:  T_Y^k(y)\in\YY^{m-k}\cap\{y\in Y:S_n(y)=k\})=\sum_{k=1}^{m}p^{m-k}\nu(\{y\in Y:S_n(y)=k\})
\]
and the conclusion follows.~\end{proof}

\begin{lemma}
\label{lemma-taumsurv}
Assume the set up of Subsection~\ref{subsec-trans}, in particular~\eqref{eq-tail}.  Assume that ~\eqref{eq-ap} holds.
Suppose that
$\nu(\tau>n)=Cn^{-\beta}(1+o(1))$ with $\beta\in(0,1)$. Then
\[
\nu(\tt_{m}\ge n| \YY^{m})\to \P(\mathcal{Y}_\beta \le t).
\]
\end{lemma}

\begin{proof}Since we condition on the survivor set $ \YY^{m}$,
the required   argument is standard
and we sketch it here only for completeness.
It can be regarded as a straightforward modification of, for instance,
the argument used in the proof of the central limit theorem for Markov chains with quasi stationary distributions~\cite[Theorem 3.4]{CLM}.

Let $\RR$ be the matrix with entries given by~\eqref{eq-matrix-indh}. Let $r=d\nu/d Leb$ and note that  $r$ is constant
on $Y=\cup_{k\ge 0}C_k$. Also, we note that in the set up of Subsection~\ref{subsec-trans},
$\RR r=p r$ and $\RR (e^{i\theta\tt}r)=p  e^{i\theta\tt}r$, $\theta\in [-\pi,\pi)$. Next, let $\tilde R=p^{-1}\RR$ be the normalization of $\RR$ and note that
for $m\ge 0$,
\[
\mathbb{E}_\nu(e^{i(\theta/m^{1/\beta})\tt_m}|\YY^m)=\int_{\YY^m}\tilde R^m r e^{i(\theta/m^{1/\beta})\tt_m}\, dLeb=p^{-m}\int_{\YY^m} e^{i(\theta/m^{1/\beta})\tt_m}\, d\nu.
\]
For $m=1$, using the notation in~\eqref{eq-tail}, $\mathbb{E}_\nu(e^{i\theta\tt}|\YY)=p^{-1}\sum_{n=0}^\infty \ff_n e^{in\theta}=\sum_{n=0}^\infty f_n e^{in\theta}$.
Since by assumption, $\sum_{j>n}f_j=\nu(\tau>n)=Cn^{-\beta}(1+o(1))$ with $\beta\in(0,1)$, as $\theta\to 0$
\[
1-\mathbb{E}_\nu(e^{i\theta\tt}|\YY)=C_\beta\theta^\beta(1+o(1)),
\]
where $C_\beta$ is a constant that depends only on $C$ and $\beta$ (see, for instance,~\cite{Feller66}). Thus,
\[
\mathbb{E}_\nu(e^{i(\theta/m^{1/\beta})\tt_m}|\YY^m)=\exp(m\log (\mathbb{E}_\nu(e^{i(\theta/m^{1/\beta})\tt}|\YY)))=e^{C_\beta\theta^\beta}(1+o(1)),
\]as required.~\end{proof}

\appendix

\section{A result used in Proofs of Propositions~\ref{prop-main} and~\ref{prop-main2}}

In this appendix, we use ``big O'' and $\ll$ notation interchangeably, writing
$A_n=O(a_n)$ or $A_n\ll a_n$ as $n\to\infty$ if there is a constant
$C>0$ such that $\|A_n\|\le Ca_n$ for all $n\ge1$ (for $A_n$ operators
and $a_n\ge0$ scalars). 

\begin{lemma}\label{lemma-abstr}
Let $A(z)$ and $B(z)$ be operator valued functions on some function space with norm $\|\,\|$,
analytic  on $\D$ such that $A(1)=B(1)=0$. Suppose that the coefficients $A_n, B_n$ of
$A(z), B(z)$, $z\in\D$ are such that $\|A_n\|\ll \|B_n\|\ll n^{-(\beta+1)}$ for some $\beta>0$.

Define $C(z)=(1-z)^{-1}A(z)B(z)$, $z\in\D$. Then the coefficients of $C(z)$, $z\in\D$ satisfy
$\|C_n\|\ll n^{-2\beta}$ if $\beta<1$, $\|C_n\|\ll (\log n)n^{-2}$ if $\beta=1$
 and $\|C_n\|\ll n^{-(\beta+1)}$ if $\beta\geq 1$.
\end{lemma}

\begin{proof}During this proof $A', B', C'$ denote the first derivatives of $A, B, C$
and $A'_n, B'_n, C'_n$ denote the $n$-th coefficient of these functions on
$\D$. 

Clearly, $\|C_n\|\ll n^{-1}\|C'_n\|$. It remains to estimate the coefficients
of $C'(z)$. An easy calculations shows that
\begin{align*}
C'(z)=A'(z)\frac{B(z)}{1-z}+\frac{A(z)}{1-z}B'(z)+\frac{A(z)}{1-z}\frac{B(z)}{1-z}.
\end{align*}
Since $B(1)=0$,
$(1-z)^{-1}B(z)=\sum_{n}(\sum_{j\geq n}B_j) z^j$. Hence, the coefficients (in norm $\|\,\|$)
of $(1-z)^{-1}B(z)$ are $O(n^{-\beta})$. Similarly, the coefficients
of $(1-z)^{-1}A(z)$ are $O(n^{-\beta})$. Also, by assumption, $\|A'_n\|\ll\|B_n'\|\ll n^{-\beta}$. 
Putting these together by convolving the coefficients of the factors corresponding to
the three terms in the expression above of $C'(z)$,
\begin{align*}
\|C_n'\|\ll\begin{cases}
n^{1-2\beta}, & 0<\beta\leq\frac12, \\ n^{-(2\beta-1)}, & \frac12<\beta<1,
\\ (\log n)n^{-1}, &\beta=1,
 \\ n^{-\beta}, & \beta>1,
\end{cases}
\end{align*}
and the conclusion follows.~\end{proof}

\end{document}